\documentclass{article}

\date{}

\usepackage{amsmath,amssymb,amsthm,nccmath,mathtools}
\usepackage{graphicx}
\usepackage{comment}
\usepackage{color}
\usepackage{enumitem}

\usepackage{geometry}

\newtheorem{lemma}{Lemma}
\newtheorem{theorem}{Theorem}
\newtheorem{proposition}{Proposition}
\newtheorem{corollary}{Corollary}
\newtheorem{assumption}{Assumption}

\theoremstyle{remark}
\newtheorem{remark}{Remark}

\theoremstyle{definition}
\newtheorem{definition}{Definition}

\emergencystretch=\maxdimen
\setlist{leftmargin=0pt,itemindent=15pt,labelwidth=8pt,labelsep=7pt,listparindent=15pt}

\title{Semi-Global Exponential Stability of Augmented Primal-Dual Gradient Dynamics for Constrained Convex Optimization}

\author{Yujie Tang, Guannan Qu and Na Li
\thanks{Y. Tang and N. Li are with the School of Engineering and Applied Sciences, Harvard University, Cambridge, MA 02138. G. Qu is with the Department of Computing and Mathematical Sciences, California Institute of Technology, Pasadena, CA 91125.
\newline
\indent Emails: {\tt\small yujietang@seas.harvard.edu, gqu@caltech.edu, nali@seas.harvard.edu}}}

\begin{document}

\maketitle

\begin{abstract}
Primal-dual gradient dynamics that find saddle points of a Lagrangian have been widely employed for handling constrained optimization problems. Building on existing methods, we extend the augmented primal-dual gradient dynamics (Aug-PDGD) to incorporate general convex and nonlinear inequality constraints, and we establish its semi-global exponential stability when the objective function is strongly convex. We also provide an example of a strongly convex quadratic program of which the Aug-PDGD fails to achieve global exponential stability. Numerical simulation also suggests that the exponential convergence rate could depend on the initial distance to the KKT point.

\vspace{5pt}
\end{abstract}

\section{Introduction}

This paper introduces and analyzes a version of primal-dual gradient dynamics, called \emph{augmented primal-dual gradient dynamics} (Aug-PDGD, see \eqref{eq:AugPDGD}), which aims to solve the following smooth convex optimization problem:
\begin{equation*}
\begin{aligned}
\min_{x}\quad & f(x) \\
\textrm{s.t.}\quad & g(x)\leq 0, \\
&Ax = b.
\end{aligned}
\end{equation*}

Since its first introduction in \cite{kose1956solutions,arrow1958studies}, primal-dual gradient dynamics (PDGD) have been applied for solving optimization problems in various applications, including power system operation \cite{zhao2014design,chen2018distributed}, communication networks \cite{chiang2007layering}, distributed optimization \cite{wang2011control,lee2016distributed,cortes2019distributed}, etc. Theoretical studies of PDGD can date back to the 1950s and have been of continuing interest to researchers. The early works \cite{kose1956solutions,arrow1958studies} have already focused on the convergence of projected PDGD for constrained convex programs. Further results on convergence and asymptotic stability of projected PDGD appeared in \cite{rockafellar1971saddle,flaam1989approximating,venets1985continuous}. Especially, \cite{venets1985continuous} suggested that, when projected PDGD is used for solving smooth convex programs, global convergence to the set of KKT points can be ensured when the objective function is strictly convex. The paper \cite{goebel2017stability} extended the results in \cite{venets1985continuous} to nonsmooth problems and also considered robustness of asymptotic stability. Other recent works on global asymptotic stability of PDGD include \cite{feijer2010stability,cherukuri2016asymptotic,cherukuri2017saddle,cherukuri2018role,dhingra2018proximal}.
For instance, \cite{cherukuri2016asymptotic} proved global asymptotic stability of projected PDGD with strictly convex objectives based on a version of the invariance principle for discontinuous dynamical systems, and \cite{cherukuri2017saddle} established conditions for asymptotic stability for more general saddle point problems. \cite{cherukuri2018role} showed global asymptotic stability of projected PDGD for locally strongly convex-concave Lagrangian.

Exponential stability is a desirable property both theoretically and in practice. Particularly, given a continuous-time dynamics that is exponentially stable, one can obtain a discrete-time iterative algorithm through explicit Euler discretization that achieves linear convergence for sufficiently small step sizes under appropriate conditions \cite{stuart1994numerical,stetter1973analysis}. It is well-known that for unconstrained convex optimization, when the objective function is smooth and strongly convex, the projected gradient dynamics achieves global exponential stability, and as the discrete-time counterpart, the projected gradient descent algorithm achieves global linear convergence \cite{nesterov2004introductory}, and \cite{necoara2018linear} showed that the condition of strong convexity could be relaxed. In the context of primal-dual gradient dynamics and constrained convex optimization, it is known that local exponential stability can be established by resorting to spectral bounds of saddle matrices \cite{benzi2005numerical,shen2010eigenvalue}. Regarding global exponential stability, convex analysis provides guarantees for PDGD with a strongly-convex-strongly-concave Lagrangian, which, however, does not directly apply to PDGD for constrained convex programs as the resulting Lagrangian is not strongly concave in the dual variable. \cite{niederlander2016exponentially} and \cite{cortes2019distributed} studied saddle-point-like dynamics and proved global
exponential stability when applying such dynamics to equality constrained convex optimization problems. 
Several papers proposed primal-dual gradient dynamics that can be applied to smooth strongly convex programs with affine constraints $Fx\leq \upsilon$, and showed global exponential stability when $F$ has full row rank: \cite{qu2019exponential} introduced the augmented primal-dual gradient dynamics and showed global exponential stability by a Lyapunov-based approach; \cite{dhingra2018proximal} proposed the proximal gradient flow, and showed global exponential stability by employing the theory of integral quadratic constraints; \cite{ding2019global} is a continuing work of \cite{dhingra2018proximal} that provided a quadratic Lyapunov function yielding less conservative convergence rate estimates than \cite{qu2019exponential}; \cite{bansode2019exponential} showed that by utilizing a Riemann metric, the resulting projected primal-dual gradient dynamics achieves global exponential stability.

This work is an extension of the results in \cite{qu2019exponential}. The main contributions are summarized as follows.
\begin{itemize}
\item We extend the Aug-PDGD algorithm in \cite{qu2019exponential} to a very general setting of smooth convex optimization, where the constraint functions can be convex and nonlinear. The use of augmented Lagrangian for handling inequality constraints results in a continuous dynamical system, which is different from projection-based primal-dual gradient dynamics.

\item We generalize and improve on the approach in \cite{qu2019exponential} to show that when the objective function $f$ is strongly convex, the Aug-PDGD achieves semi-global exponential stability \cite{sastry1999nonlinear}, i.e., regardless of the initial point, the distance to the optimal solution decays exponentially, but the exponential convergence rate could depend on the initial point. Specifically, we show that given the initial point, one can find $\beta>0$ such that the distance to the optimal solution decays as $O(e^{-\beta t})$. Furthermore, we provide an upper bound on $\beta$ to quantitatively characterize the dependence of the exponential convergence rate on the initial point, which is non-decreasing as the initial point becomes closer to the equilibrium point.
The proof is based on a quadratic Lyapunov function that has non-zero off-diagonal terms.

Compared to existing works \cite{niederlander2016exponentially,cortes2019distributed,qu2019exponential,dhingra2018proximal,ding2019global,bansode2019exponential}, we consider a more general setting where the inequality constraints are convex and nonlinear, and the gradient vectors of the inequality and equality constraint functions at the optimal point need not be linearly independent. Consequently, our analysis and results are applicable to a wider range of practical problems.

\item We provide an example of a strongly convex quadratic program of which the Aug-PDGD fails to achieve global exponential stability. The example has an inactive affine constraint whose gradient is not linearly independent with the gradient of the equality constraint. This example suggests that semi-global exponential stability might be the strongest stability behavior we can establish for Aug-PDGD for general smooth strongly convex optimization with constraints.
\end{itemize}

\vspace{2pt}
\noindent\textit{Notation.}
For any $x\in\mathbb{R}$, we denote $[x]_+\coloneqq\max\{x,0\}$. For any $x\in\mathbb{R}^p$, we use $x\geq 0$ to mean that all entries of $x$ are nonnegative. For any real symmetric matrices $P$ and $Q$, $P\succeq Q$ and $Q\preceq P$ mean that $P-Q$ is positive semidefinite; similarly $P\succ Q$ and $Q\prec P$ mean that $P-Q$ is positive definite. For any $x\in\mathbb{R}^p$, we use $\|x\|$ to denote the $\ell_2$ norm of $x$, and denote $\|x\|_Q\coloneqq\sqrt{x^TQx}$ when $Q$ is a positive definite matrix. For any matrix $Q$, we use $\|Q\|$ to denote the spectral norm of $Q$. The identity matrix will be denoted by $I$. The standard basis of $\mathbb{R}^{p}$ will be denoted by $\{e_i\}_{i=1}^p$. For a finite set $S$, we use $|S|$ to denote the number of elements in $S$.

\section{Augmented Primal-Dual Gradient Dynamics}

In this section, we present a more detailed description of the augmented primal-dual gradient dynamics (Aug-PDGD) and introduce some preliminary results regarding its equilibrium point and trajectory behavior.

Recall that the main problem is formulated as
\begin{equation}\label{eq:main_problem}
\begin{aligned}
\min_{x}\quad & f(x) \\
\textrm{s.t.}\quad & g(x)\leq 0, \\
&Ax = b.
\end{aligned}
\end{equation}
Here $f:\mathbb{R}^n\rightarrow\mathbb{R}$ is continuously differentiable and convex, $g:\mathbb{R}^n\rightarrow\mathbb{R}^{m_I}$ is continuously differentiable and convex, and $A\in\mathbb{R}^{m_E\times n}$, $b\in\mathbb{R}^{m_E}$.

We introduce the augmented Lagrangian of \eqref{eq:main_problem} formulated as \cite{rockafellar1970new,bertsekas1996constrained}
\begin{equation}\label{eq:aug_Lagrangian}
L_{\rho}(x,\lambda,\nu)\coloneqq
f(x)+\Theta_\rho(x,\lambda)+\nu^T(Ax-b)
\end{equation}
where the auxiliary function $\Theta_\rho$ is defined by
$$
\begin{aligned}
\Theta_\rho(x,\lambda)
&\coloneqq\sum_{i=1}^{m_I}\frac{[\rho g_i(x)+\lambda_i]_+^2-\lambda_i^2}{2\rho}.
\end{aligned}
$$
The domain of the augmented Lagrangian $L_\rho(x,\lambda,\nu)$ is
$\{(x,\lambda,\nu)\in\mathbb{R}^n\times\mathbb{R}^{m_I}\times\mathbb{R}^{m_E}:\lambda\geq 0\}$. It can be checked that $\Theta_\rho(x,\lambda)$ is convex in $x$ and concave in $\lambda$, and that $\Theta_\rho(x,\lambda)$ is continuously differentiable. The augmented primal-dual gradient dynamics is then given by
\begin{subequations}\label{eq:AugPDGD}
\begin{align}
\dot{x}(t)
=\,& -\nabla_x L_\rho(x(t),\lambda(t),\nu(t))
\nonumber \\
=\,&
-\nabla f(x(t)) -A^T\nu(t)
-\sum_{i=1}^{m_I}[\rho g_i(x(t))+\lambda_i(t)]_+\nabla g_i(x(t)), \\
\dot{\lambda}(t)
=\, &
\nabla_\lambda L_\rho(x(t),\lambda(t),\nu(t)) \nonumber \\
=\,&
\sum_{i=1}^{m_I} \frac{[\rho g_i(x(t))+\lambda_i(t)]_+-\lambda_i(t)}{\rho}e_i, \\
\dot{\nu}(t)
=\, &
\nabla_\nu L_\rho(x(t),\lambda(t),\nu(t)) \nonumber \\
=\,&
Ax(t)-b.
\end{align}
\end{subequations}
We shall also denote $z(t)=(x(t),\lambda(t),\nu(t))$ for brevity.

Suppose $(x(t),\lambda(t),\nu(t)),\ t\geq 0$ is a differentiable trajectory that satisfies the Aug-PDGD \eqref{eq:AugPDGD} for all $t\geq 0$. The following proposition summarizes some preliminary results on the equilibrium and trajectory behavior of the Aug-PDGD, whose proof is rather straightforward which we omit here.
\begin{proposition}\label{proposition:AugPDGD_basic}
\begin{enumerate}
\item {\cite{bertsekas1996constrained}} A primal-dual pair is an equilibrium point of the Aug-PDGD \eqref{eq:AugPDGD} if and only if it is a KKT point of \eqref{eq:main_problem}.
\item Suppose $\lambda(0)\geq 0$. Then $\lambda(t)\geq 0$ for all $t\geq 0$.
\end{enumerate}
\end{proposition}
\vspace{-8pt}
\begin{remark}
The augmented primal-dual gradient dynamics differs from the standard projected primal-dual gradient dynamics by employing $\Theta_\rho(x,\lambda)$ instead of $\lambda^T g(x)$ in constructing the Lagrangian and by the lack of projection of $\lambda$ onto the nonnegative orthant. This form of augmented Lagrangian was first proposed in \cite{rockafellar1970new} from the perspective of duality theory, and its properties and applications in optimization have been studied in the literature (see \cite{bertsekas1996constrained} and the references therein). As Proposition \ref{proposition:AugPDGD_basic} shows, the KKT point of \eqref{eq:main_problem} coincides with the equilibrium of \eqref{eq:AugPDGD}, and as long as $\lambda(0)\geq 0$, $\lambda(t)$ will remain nonnegative even if there is no explicit projection onto the nonnegative orthant. One advantage of \eqref{eq:AugPDGD} is that its right-hand sides are all continuous in $(x,\lambda,\nu)$, unlike the standard projected primal-dual gradient dynamics in which projection introduces discontinuity.

On the other hand, we point out that the augmented Lagrangian \eqref{eq:aug_Lagrangian} is not strongly concave in $(\lambda,\nu)$, and we cannot apply the fact that primal-dual gradient dynamics for strongly-convex-strongly-concave Lagrangians achieve (global) exponential stability in our situation.
\end{remark}
\vspace{-10pt}
\begin{remark}
In \cite{qu2019exponential}, an additional parameter $\eta>0$ that scales the dual gradients was introduced in the Aug-PDGD
\begin{equation}\label{eq:AugPDGD_scaled}
\begin{aligned}
\dot{x} &=-\nabla_x L_\rho(x,\lambda,\nu), \\
\dot{\lambda} &=\eta\nabla_\lambda L_\rho(x,\lambda,\nu), \\
\dot{\nu} &=\eta\nabla_\nu L_\rho(x,\lambda,\nu).
\end{aligned}
\end{equation}
In this paper we neglect this parameter, as one can recover \eqref{eq:AugPDGD_scaled} by scaling the constraint functions and dual variables as
\begin{align*}
g& \rightarrow \sqrt{\eta} g, &
A& \rightarrow \sqrt{\eta} A, &
b& \rightarrow\sqrt{\eta}b, \\
\lambda &\rightarrow\frac{\lambda}{\sqrt{\eta}}, &
\nu & \rightarrow\frac{\nu}{\sqrt{\eta}}, &
\rho& \rightarrow\frac{\rho}{\eta}.
\end{align*}
\end{remark}

We have the following result that guarantees the boundedness of the trajectory, which follows from the fact that the Aug-PDGD \eqref{eq:AugPDGD} is a special case of continuous convex-concave saddle point dynamics.
\begin{lemma}\label{lemma:bounded_traj}
Suppose $z^\ast=(x^\ast,\lambda^\ast,\nu^\ast)$ is a KKT point of \eqref{eq:main_problem}. Then for all $t\geq 0$, we have
$
\|z(t)-z^\ast\|
\leq
\|z(0)-x^\ast\|
$.
\end{lemma}

\section{Main Results}

\subsection{On the Notion of Semi-Global Exponential Stability}
We first introduce the notion of semi-global exponential stability:
\begin{definition}[\cite{sastry1999nonlinear}]\label{definition:semi_global_expo}
Consider the following autonomous dynamical system
\begin{equation}\label{eq:general_dyn}
\dot{z}(t)=\phi(z(t)),
\end{equation}
and let $z_e$ be an equilibrium point of \eqref{eq:general_dyn}. We say that $z_e$ is a \emph{semi-globally exponentially stable} equilibrium point, if for any $h>0$, there exist some $M>0$ and $\beta>0$ such that whenever the initial point $z(0)$ satisfies $\|z(0)-z_e\|\leq h$, the corresponding solution $z(t)$ to \eqref{eq:general_dyn} satisfies
\begin{equation}\label{eq:def_semi_global_exp_stabl}
\|z(t)-z_e\|\leq M e^{-\beta t} \|z(0)-z_e\|,\quad
\forall t\in[0,+\infty).
\end{equation}
We say that $z_e$ is a \emph{globally exponentially stable} equilibrium point if, in addition, the constants $M$ and $\beta$ in \eqref{eq:def_semi_global_exp_stabl} do not depend on $h$.
\end{definition}

Loosely speaking, if $z_e$ is a semi-globally exponentially stable equilibrium, then for any trajectory $z(t)$, the (normalized) distance to the equilibrium $\|z(t)\!-\!z_e\|/\|z(0)\!-\!z_e\|$ can be upper bounded by a decaying exponential function of the form $M e^{-\beta t}$, but the decaying rate $\beta$ and the factor $M$ can depend on the initial distance $\|z(0)-z_e\|$. When $M$ and $\beta$ can be chosen to be independent of the initial distance $\|z(0)-z_e\|$, we drop the prefix ``semi'' and say that $z_e$ is a globally exponentially stable equilibrium.

Note that Definition \ref{definition:semi_global_expo} does not impose explicit restriction on the initial point $z(0)$. On the other hand, it is known from convex analysis that the dual variable associated with inequality constraints always lies the nonnegative orthant, and Proposition \ref{proposition:AugPDGD_basic} guarantees $\lambda(t)\geq 0$ for all $t>0$ as long as $\lambda(0)\geq 0$. Therefore in this paper, we impose the restriction $\lambda(0)\geq 0$ for the Aug-PDGD \eqref{eq:AugPDGD}, and still say that the Aug-PDGD achieves semi-global exponential stability when the conditions in Definition \ref{definition:semi_global_expo} are satisfied under the restriction $\lambda(0)\geq 0$.

As will be shown later, the KKT point of \eqref{eq:main_problem} is a semi-globally exponentially stable equilibrium point of the Aug-PDGD \eqref{eq:AugPDGD} when $f$ is strongly convex.

\subsection{Assumptions and Additional Notations}
Next we introduce some assumptions and additional notations that will be used throughout the paper.

\begin{assumption}\label{assumption:strongly_convex}
The problem \eqref{eq:main_problem} is feasible, and the objective function $f(x)$ is $\mu$-strongly convex. Consequently
\begin{equation}\label{eq:quad_grad_growth}
(x-x^\ast)^T(\nabla f(x)-\nabla f(x^\ast))\geq \mu \|x-x^\ast\|^2,
\quad\forall x\in\mathbb{R}^n,
\end{equation}
where $x^\ast$ denotes the unique solution to \eqref{eq:main_problem}.
\end{assumption}

\begin{remark}
We shall see that our proof utilizes the strong convexity of $f(x)$ purely through the inequality \eqref{eq:quad_grad_growth}. This indicates that Assumption~\ref{assumption:strongly_convex} can be replaced by a weaker version using the quadratic gradient growth condition \cite{necoara2018linear}, which has been introduced as a relaxed condition of strong convexity for establishing linear convergence of gradient descent. Nevertheless, we stick to the strong convexity assumption for conceptual simplicity.
\end{remark}

We then impose further assumptions on the constraints of the problem \eqref{eq:main_problem}. Let the active set at $x^\ast$ be denoted by $\mathcal{I}\coloneqq \{i:g_i(x^\ast)=0\}$, and let its complement be denoted by $\mathcal{I}^c\coloneqq \{1,\ldots,m_I\}\backslash\mathcal{I}$. The Jacobian matrix of $g(x)$ at $x^\ast$ will be denoted by
$$
J\coloneqq
\begin{bmatrix}
\nabla g_1(x^\ast) & \cdots &
\nabla g_{m_I}(x^\ast)
\end{bmatrix}^T.
$$
We use $J_{\mathcal{I}}$ to denote the matrix formed by the rows of $J$ whose indices are in $\mathcal{I}$.

\begin{assumption}\label{eq:assumption_LICQ}
The linear independence constraint qualification (LICQ) \textup{\cite{wachsmuth2013licq}} holds at $x^\ast$, i.e., the row vectors of $J_{\mathcal{I}}$ and $A$ are linearly independent. We also assume $|\mathcal{I}|+m_E>0$.
\end{assumption}

Consequently, there exist unique optimal Lagrange multipliers $\lambda^\ast$, $\nu^\ast$ such that $z^\ast=(x^\ast,\lambda^\ast,\nu^\ast)$ satisfies the KKT conditions. We denote
$$
\kappa \coloneqq \lambda_{\min}\left(\!\begin{bmatrix}
J_{\mathcal{I}} \\
A
\end{bmatrix}\begin{bmatrix}
J_{\mathcal{I}} \\
A
\end{bmatrix}^{\!T} \!\right).
$$
Assumption \ref{eq:assumption_LICQ} then implies $\kappa>0$.

Next, we introduce the quantity $d_0$ defined as the distance between the initial primal-dual pair $z(0)$ and the KKT point $z^\ast$:
$$
d_0 \coloneqq
\|z(0)-z^\ast\|
=\!\left(\|x(0)-x^\ast\|^2+\|\lambda(0)-\lambda^\ast\|^2
+\|\nu(0)-\nu^\ast\|^2\right)^{1/2}.
$$
Lemma \ref{lemma:bounded_traj} then implies that for all $t\geq 0$,
$$
\begin{aligned}
d_0\geq\ & \left(\|x(t)-x^\ast\|^2+\|\lambda(t)-\lambda^\ast\|^2
+\|\nu(t)-\nu^\ast\|^2\right)^{1/2} \\
\geq\ &\max\{\|x(t)-x^\ast\|,
\|\lambda(t)-\lambda^\ast\|,
\|\nu(t)-\nu^\ast(t)\|\}.
\end{aligned}$$

\begin{assumption}
$\nabla f(x)$ is $\ell$-Lipschitz continuous. Also, for any initial distance $d_0$, there exist $L_{g,i}\geq 0,M_{g,i}\geq 0$ such that $\|\nabla g_i(x)\|\leq L_{g,i}$ and $\nabla g_i(x)$ is $M_{g,i}$-Lipschitz continuous over $x\in \{y:\|y-x^\ast\|\leq d_0\}$ for each $i=1,\ldots,m_I$.
\end{assumption}
Without loss of generality, we choose $L_{g,i}$ and $M_{g,i}$ to be non-decreasing when $d_0$ increases. We also denote
$$
L_g\coloneqq \sqrt{\sum_{i=1}^{m_I} L_{g,i}^2},\qquad M_g\coloneqq\sqrt{\sum_{i=1}^{m_I} M_{g,i}^2}.
$$
The quantity $L_g$ can be viewed as an upper bound on the Frobenius norm (and consequently the spectral norm) of the Jacobian matrix of $g(x)$. Also note that $L_g$ and $M_g$ are non-decreasing when $d_0$ increases.

\begin{lemma}
Let $\lambda\geq 0$ satisfy $\|\lambda-\lambda^\ast\|\leq d_0$. Then for any $x_1,x_2$ such that $\|x_1-x^\ast\|\leq d_0$ and $\|x_2-x^\ast\|\leq d_0$,
$$
\left\|
\nabla_x\Theta_\rho(x_1,\lambda)-
\nabla_x\Theta_\rho(x_2,\lambda)\right|
\leq M_{\Theta}\|x_1-x_2\|,
$$
where
$M_{\Theta}\coloneqq\rho L_g^2
+(\rho L_gd_0+d_0+\|\lambda^\ast\|)M_g.
$
\end{lemma}
\begin{proof}
By a direct calculation, we have
$$
\begin{aligned}
&\left\|\nabla_x\Theta_\rho(x_1,\lambda)-
\nabla_x\Theta_\rho(x_2,\lambda)\right\| \\
=\ &
\Bigg\|\sum_{i=1}^{m_I} \big(([\rho g_i(x_1)+\lambda_i]_+-[\rho g_i(x_2)+\lambda_i]_+)\nabla g_i(x_1)
+[\rho g_i(x_2)+\lambda_i]_+(\nabla g_i(x_1)-\nabla g_i(x_2)\big)\Big\| \\
\leq\ &
\sum_{i=1}^{m_I}
\big(\|[\rho g_i(x_1)+\lambda_i]_+-[\rho g_i(x_2)+\lambda_i]_+\|\|\nabla g_i(x_1)\|
+[\rho g_i(x_2)+\lambda_i]_+
\|\nabla g_i(x_1)
-\nabla g_i(x_2)\|\big) \\
\leq\ &
\sum_{i=1}^{m_I}
\big(\|\rho g_i(x_1)-\rho g_i(x_2)\| L_{g,i}
+(\rho L_{g,i}\|x_2-x^\ast\|
+\lambda_i) M_{g,i}\|x_1-x_2\|\big) \\
\leq\ &
\sum_{i=1}^{m_I} \rho L_{g,i}^2 \|x_1-x_2\|
+(\rho L_gd_0+\|\lambda\|) M_{g}\|x_1-x_2\| \\
\leq\ &
\left(\rho L_g^2
+(\rho L_gd_0+d_0+\|\lambda^\ast\|) M_g\right)\|x_1-x_2\|,
\end{aligned}
$$
which gives the desired result.
\end{proof}
This lemma shows that, $M_{\Theta}$ can be viewed as the Lipschitz constant of $\nabla_x\Theta_\rho(x,\lambda)$ with respect to $x$ in the region $\left\{x\in\mathbb{R}^n:\|x-x^\ast\|\leq d_0\right\}$.

Note that under the above assumptions, asymptotic stability of $(x^\ast,\lambda^\ast,\nu^\ast)$ can be guaranteed by existing results on primal-dual gradient dynamics \cite{goebel2017stability,dhingra2018proximal}. In the following, we present our main results characterizing semi-global exponential stability of the Aug-PDGD.

\subsection{The Main Results}

\begin{theorem}\label{theorem:main}
Suppose $\lambda(0)\geq 0$. Under Assumptions 1, 2 and 3, the trajectory $z(t)=(x(t),\lambda(t),\nu(t))$ of the augmented primal-dual gradient dynamics \eqref{eq:AugPDGD} satisfies
\begin{equation}\label{eq:semi_exponen_stab}
\left\|
z(t)-z^\ast\right\|
\leq M_\beta \cdot e^{-\beta t} \left\|
z(0)-z^\ast\right\|.
\end{equation}
Here $\beta$ is any strictly positive constant satisfying
\begin{subequations}\label{eq:main_thm_cond}
\begin{equation}\label{eq:main_thm_cond_1}
\beta \leq
\frac{\kappa\delta_{\min}}{46\rho(L_g^2+\|A\|^2)}
\end{equation}
and
\begin{equation}\label{eq:main_thm_cond_2}
\begin{aligned}
\frac{\kappa\mu}{4\beta}-4\beta^2
\geq\,&
\|A\|^2+L_g^2+\frac{\kappa}{4}
+(\ell+M_{\Theta})\!\left(\mu \!+\! M_{\Theta} \!+\! \mfrac{1}{\rho}\right)
\!+\! \mfrac{1}{2\rho^2}
,
\end{aligned}
\end{equation}
\end{subequations}
where
\begin{equation}\label{eq:def_delta_min}
\delta_{\min}\coloneqq
1-\left[1+\frac{\rho\cdot\max_{i\in\mathcal{I}^c} g_i(x^\ast)}{d_0}\right]_+^2;
\end{equation}
$M_\beta$ is a positive real number that depends on $\beta$ and the problem \eqref{eq:main_problem} itself, which satisfies $\lim_{\beta\rightarrow 0^+}M_\beta=1$.
\end{theorem}
\begin{corollary}\label{corollary:main}
Under the conditions of Theorem~\ref{theorem:main}, the KKT point $z^\ast=(x^\ast,\lambda^\ast,\nu^\ast)$ is a semi-globally exponentially stable equilibrium of the Aug-PDGD \eqref{eq:AugPDGD}.
\end{corollary}
We make some discussion on the interpretation and implications of the theorem and the corollary:
\begin{itemize}
\item {\bf Existence of $\beta$.} Since $\mathcal{I}^c$ consists of the indices of all inactive constraints, we see that $g_i(x^\ast)$ is strictly negative for all $i\in\mathcal{I}^c$, and therefore $\delta_{\min}$ is strictly positive. Furthermore, the left-hand side of \eqref{eq:main_thm_cond_2} is a decreasing function of $\beta$ that goes to $+\infty$ as $\beta\rightarrow 0^+$. Therefore, a strictly positive constant $\beta$ satisfying \eqref{eq:main_thm_cond} will always exist for any $(x(0),\lambda(0),\nu(0))$ with $\lambda(0)\geq 0$.

\item {\bf Semi-global exponential stability.} Since $\delta_{\min}$ is positive but decreases to zero as $d_0\rightarrow+\infty$, and $L_g$ and $M_{\Theta}$ are non-decreasing as $d_0$ increases, we see that the bound provided by \eqref{eq:main_thm_cond} on the exponential convergence rate $\beta$ depends on the initial distance $d_0$ and decreases to zero as $d_0\rightarrow+\infty$. Therefore, Theorem \ref{theorem:main} does not guarantee the existence of a universal exponential convergence rate, and only semi-global exponential stability can be established in Corollary~\ref{corollary:main} (see Step 3 in Section 4). 
This is different from the situations with only equality constraints or where $g(x)=Fx \!-\! \upsilon$ with $F$ and $A$ having linearly independent rows \cite{niederlander2016exponentially,cortes2019distributed,dhingra2018proximal,qu2019exponential,ding2019global}.
In Section \ref{sec:counter_example}, we will present a counterexample showing that Aug-PDGD does not achieve global exponential stability in the setting discussed here.

\item {\bf Constraints that lead to faster convergence.} By definition we have $\kappa\leq L_g^2+\|A\|^2$, and we can interpret the ratio $(L_g^2+\|A\|^2)/\kappa$ as the ``condition number'' of the constraints. Then the bounds \eqref{eq:main_thm_cond} suggest that better conditioned constraints can lead to faster convergence. We also see that when $M_g$ is smaller and $g(x)$ is closer to being affine, the bound \eqref{eq:main_thm_cond_2} will also be larger.

\end{itemize}

\section{Proof of Theorem~\ref{theorem:main} and Corollary~\ref{corollary:main}}\label{sec:main_proof}

For notational simplicity we suppress the dependence on $t$ and use $(x,\lambda,\nu)$ to denote the trajectory of \eqref{eq:AugPDGD}.

We let $c=4\kappa^{-1}\beta$ where $\beta$ satisifes the conditions \eqref{eq:main_thm_cond}, and define\footnote{
This quadratic Lyapunov function has been introduced in \cite{qu2019exponential} for problems with affine constraints. We refer the readers to \cite{qu2019exponential} for more discussions on this Lyapunov function.
}
$$
P_c \! \coloneqq \! \begin{bmatrix}
 I \! &  c J^T \! & \! c A^T \\
 c J \! & I \! & \! 0 \\
 c A \! & 0 \! & \! I
\end{bmatrix},
\ \ \ \ 
V_c\coloneqq \frac{1}{2} \!\begin{bmatrix}
x-x^\ast \\
\lambda-\lambda^\ast \\
\nu - \nu^\ast
\end{bmatrix}^T \!\!\!
P_c\!
\begin{bmatrix}
x-x^\ast \\
\lambda-\lambda^\ast \\
\nu - \nu^\ast
\end{bmatrix},
$$
where we remind the readers that $J$ denotes the Jacobian matrix of $g(x)$ at $x^\ast$. The main idea is to prove that the matrix $P_c$ is positive definite (so that $V_c$ serves as a quadratic Lyapunov function) and that
$
\dot{V}_c
\leq-2\beta V_c
$,
which then leads to \eqref{eq:semi_exponen_stab} by Gr\"{o}nwall's inequality and by taking
$
M_\beta \coloneqq
\sqrt{\left\|P_c\right\|
\left\|P_c^{-1}\right\|}.
$

\vspace{2pt}
\noindent\textbf{Step 1: Prove that $P_c$ is positive definite.}
By the condition \eqref{eq:main_thm_cond_2} and the fact that $\mu\leq \ell$, we have
$$
\frac{\kappa\ell}{4\beta}
>\frac{\kappa\mu}{4\beta}-4\beta^2
>(\ell+M_{\Theta})(\mu+M_{\Theta}+\rho^{-1})
>\rho^{-1}\ell,
$$
which implies $\rho^{-1}<\kappa/(4\beta)$. Then by \eqref{eq:main_thm_cond_1} and the fact that $\delta_{\min}\leq 1$, we have
$$
\beta\leq \frac{\kappa}{46\rho(L_g^2+\|A\|^2)}
< \frac{\kappa^2}{184\beta (L_g^2+\|A\|^2)},
$$
and consequently
$
c^2=16\kappa^{-2}\beta^2<\mfrac{2}{23(L_g^2+\|A\|^2)}
$.
Then
$$
\begin{aligned}
\left\|c^2
\begin{bmatrix}
J \\ A
\end{bmatrix}^T
\begin{bmatrix}
J \\ A
\end{bmatrix}
\right\| =\ &
c^2 \|J^TJ+A^TA\|
\leq
c^2\left(\|J\|^2+\|A\|^2\right) \\
\leq\ &
c^2 \left(L_g^2+\|A\|^2\right)
< 1,
\end{aligned}
$$
and by the Schur complement condition, we see that $P_c$ is positive definite.

\vspace{2pt}
\noindent\textbf{Step 2: Prove $\dot{V}_c\leq -2\beta V_c$.}
First, we notice that
\begin{equation}\label{eq:dot_Vc_raw}
\begin{aligned}
\dot{V}_c
=&\begin{bmatrix}
x \!-\! x^\ast \\
\lambda \!-\! \lambda^\ast \\
\nu \!-\! \nu^\ast
\end{bmatrix}^{\!T} \!\!
\begin{bmatrix}
 I \! &  \! c J^T \!  &  \!\! c A^T \\
 c J \! & \! I \! & \!\! 0 \\
 c A \! & \! 0 \! & \!\! I
\end{bmatrix} \!\!
\begin{bmatrix}
-\nabla_x L_{\rho}(x,\lambda,\nu) \\
\nabla_\lambda L_{\rho}(x,\lambda,\nu) \\
\nabla_\nu L_{\rho}(x,\lambda,\nu)
\end{bmatrix} \!.
\end{aligned}
\end{equation}
Since $\Theta_\rho$ is convex in $x$ and concave in $\lambda$, we have
\begin{align*}
(x^\ast-x)^T\nabla_x\Theta_\rho(x,\lambda)
\leq\ &
\Theta_\rho(x^\ast,\lambda)-\Theta_\rho(x,\lambda), \\
(x-x^\ast)^T\nabla_x\Theta_{\rho}(x^\ast,\lambda^\ast)\leq\ &
\Theta_{\rho}(x,\lambda^\ast)-\Theta_{\rho}(x^\ast,\lambda^\ast), \\
(\lambda^\ast-\lambda)^T\nabla_\lambda\Theta_\rho(x,\lambda)
\geq\ &
\Theta_\rho(x,\lambda^\ast)-\Theta_\rho(x,\lambda),
\end{align*}
and so the diagonal terms in \eqref{eq:dot_Vc_raw} can be bounded by
$$
\begin{aligned}
\begin{bmatrix}
x-x^\ast \\
\lambda-\lambda^\ast \\
\nu-\nu^\ast
\end{bmatrix}^T
\begin{bmatrix}
-\nabla_x L_{\rho}(x,\lambda,\nu) \\
\nabla_\lambda L_{\rho}(x,\lambda,\nu) \\
\nabla_\nu L_{\rho}(x,\lambda,\nu)
\end{bmatrix}
=\,
&(x^\ast \!-\! x)^T\nabla f(x)
+(x^\ast \!-\! x)^T\nabla_x\Theta_{\rho}(x,\lambda)
+(x^\ast \!-\! x)^T A^T\nu \\
&-(\lambda^\ast-\lambda)^T\nabla_\lambda\Theta_{\rho}(x,\lambda)
-(\nu^\ast-\nu)^T(Ax-b)\\
\leq\,&
(x^\ast-x)^T\nabla f(x) +(x^\ast-x)^T A^T\nu^\ast
\\
&
+\Theta_{\rho}(x^\ast,\lambda)-\Theta_{\rho}(x,\lambda)
-(\Theta_{\rho}(x,\lambda^\ast)-\Theta_{\rho}(x,\lambda)) \\
=\,&
-(x-x^\ast)^T\nabla f(x)
+(x-x^\ast)^T\nabla f(x^\ast) \\
&-(\Theta_{\rho}(x,\lambda^\ast)-\Theta_{\rho}(x^\ast,\lambda^\ast)
-(x-x^\ast)^T\nabla_x\Theta_{\rho}(x^\ast,\lambda^\ast)) \\
&
+\Theta_{\rho}(x^\ast,\lambda)-\Theta_{\rho}(x^\ast,\lambda^\ast) \\
\leq\,&
-(x \!-\! x^\ast)^T(\nabla f(x) \!-\! \nabla f(x^\ast)) +\Theta_{\rho}(x^\ast,\lambda) \!-\! \Theta_{\rho}(x^\ast,\lambda^\ast).
\end{aligned}
$$
We define
$$
\tilde\gamma_{\lambda,i}
\coloneqq\left\{
\begin{aligned}
&1, &\quad & i\in\mathcal{I}\textrm{ or }\lambda_i=0,\\
&
[1+\rho g_i(x^\ast)/\lambda_i]_+^2,
&\quad & i\in\mathcal{I}^c\textrm{ and }
\lambda_i>0,
\end{aligned}
\right.
$$
and $\tilde\Gamma_\lambda \!\coloneqq\! \operatorname{diag}\left(\tilde\gamma_{\lambda,i}\right)_{i=1}^{m_I}$. We see that $0\preceq\tilde{\Gamma}_\lambda\preceq I$ and
$$
\Theta_\rho(x^\ast,\lambda)
-\Theta_\rho(x^\ast,\lambda^\ast)
=-\frac{1}{2\rho}(\lambda-\lambda^\ast)^T(I-\tilde\Gamma_\lambda)(\lambda-\lambda^\ast).
$$

Then we consider the off-diagonal terms of \eqref{eq:dot_Vc_raw}. For the term $\nabla_x L_{\rho}(x,\lambda,\nu)$, we have
\begin{align*}
\nabla_x L_{\rho}(x,\lambda,\nu)
=\,&
\nabla_x L_{\rho}(x,\lambda,\nu)
-\nabla_x L_{\rho}(x^\ast,\lambda^\ast,\nu^\ast) \\
=\,&
\nabla f(x)-\nabla f(x^\ast)
+\nabla_x\Theta_\rho(x,\lambda)
-\nabla_x\Theta_\rho(x^\ast,\lambda) \\
&
\!+\!\sum_{i=1}^{m_I}
\!([\rho g_i(x^\ast\!) \!+\! \lambda_i]_{\!+}
\!-\!
[\rho g_i(x^\ast\!) \!+\! \lambda^\ast_i]_{\!+})\nabla\! g_i(x^\ast\!) 
\!+\! A^T\!(\nu \!-\! \nu^\ast) \\
=\,&
\nabla f(x)-\nabla f(x^\ast)
+\nabla_x\Theta_\rho(x,\lambda)
-\nabla_x\Theta_\rho(x^\ast,\lambda) \\
&
+J^T\Gamma_\lambda(\lambda-\lambda^\ast)
+A^T(\nu-\nu^\ast),
\end{align*}
where we define
$$
\gamma_{\lambda,i}
\coloneqq\left\{
\begin{aligned}
&\frac{[\rho g_i(x^\ast)+\lambda_i]_+-[\rho g_i(x^\ast)+\lambda^\ast_i]_+}{\lambda_i-\lambda^\ast_i},&\ \  &\lambda_i\neq\lambda^\ast_i, \\
&1 &\ \  & \lambda_i=\lambda^\ast_i,
\end{aligned}
\right.
$$
and $\Gamma_\lambda\coloneqq\operatorname{diag}\left(\gamma_{\lambda,i}\right)_{i=1}^{m_I}$.

Now, if $i\in\mathcal{I}$, then $g_i(x^\ast)=0$ and $\lambda^\ast_i\geq 0$, which leads to $
\gamma_{\lambda,i}=\tilde\gamma_{\lambda,i}=1$; 
if $i\in\mathcal{I}^c$, then $g_i(x^\ast)<0$ and $\lambda^\ast_i= 0$, implying that
$$
\begin{aligned}
\frac{1-\tilde\gamma_{\lambda,i}}{1-\gamma_{\lambda,i}}
&=
\frac{1-[1+\rho g_i(x^\ast)/\lambda_i]_+^2}
{1-[1+\rho g_i(x^\ast)/\lambda_i]_+}
\geq \inf_{u\in[0,1)}\frac{1-u^2}{1-u}=1
\end{aligned}
$$
when $\lambda_i\neq \lambda^\ast_i$, and trivially $\tilde\gamma_{\lambda,i}=\gamma_{\lambda,i}=1$ when $\lambda_i=\lambda^\ast_i=0$. Thus we can see that
$I-\tilde\Gamma_\lambda\succeq I-\Gamma_{\lambda}$.

Next, we can show that
$$
\begin{aligned}
\nabla_\lambda L_{\rho}(x,\lambda,\nu)
=\,&
\nabla_\lambda L_{\rho}(x,\lambda,\nu)
-\nabla_\lambda L_{\rho}(x^\ast,\lambda^\ast,\nu^\ast) \\
=\,&
\frac{1}{\rho}
\sum_{i=1}^{m_I} (\hat{\gamma}_{x,\lambda,i}
\rho(g_i(x)-g_i(x^\ast))+(\gamma_{\lambda,i}-1)(\lambda_i-\lambda^\ast_i))e_i \\
=\,&
\hat{\Gamma}_{x,\lambda}\overline{J}_x(x-x^\ast)
+\frac{1}{\rho}({\Gamma}_\lambda-I)(\lambda-\lambda^\ast).
\end{aligned}
$$
Here we denote
$$
\hat{\gamma}_{x,\lambda,i}
\!\coloneqq\left\{\!
\begin{aligned}
&\frac{[\rho g_i(x) \!+\! \lambda_i]_+
\!-\!
[\rho g_i(x^\ast) \!+\! \lambda_i]_+}{\rho(g_i(x)-g_i(x^\ast))}, & \!\! & g_i(x) \!\neq\!  g_i(x^\ast), \\
&1, & \!\! & g_i(x) \!=\! g_i(x^\ast),
\end{aligned}
\right.
$$
which lies in $[0,1]$, and
$$
\hat\Gamma_{x,\lambda}
\coloneqq\operatorname{diag}\left(\hat\gamma_{x,\lambda,i}\right)_{i=1}^{m_I},
\ \ 
\overline{J}_x
\!\coloneqq\!\int_0^1 J_g(x^\ast\!+\!\theta(x\!-\!x^\ast))\,d\theta,
$$
where $J_g(x)$ is the Jacobian matrix of $g$ evaluated at $x$.

Summarizing the above derivations, we get
\begin{equation}\label{eq:dotVc_bound_raw}
\begin{aligned}
\dot{V}_c
\leq\,&
-(x-x^\ast)^T(\nabla f(x)-\nabla f(x^\ast))  
-\frac{1}{2\rho}(\lambda-\lambda^\ast)^T(I-\tilde{\Gamma}_\lambda)(\lambda-\lambda^\ast)
\\
&
+c(x \!-\! x^\ast)^T J^T
\Big(\hat\Gamma_{x,\lambda}\overline{J}_x
(x \!-\! x^\ast)
+\frac{1}{\rho}(\Gamma_\lambda \!-\! I)(\lambda \!-\! \lambda^\ast)
\Big)
+c(x-x^\ast)^T A^TA(x-x^\ast)
\\
&
-c\left((\lambda-\lambda^\ast)^T J
+(\nu-\nu^\ast)^T A\right)\!
\Big(\nabla f(x)-\nabla f(x^\ast)
\\
&\qquad
+\nabla_x\Theta_\rho(x,\lambda)
-\nabla_x\Theta_\rho(x^\ast,\lambda)
+J^T\Gamma_\lambda(\lambda-\lambda^\ast)
+A^T(\nu-\nu^\ast)\Big).
\end{aligned}
\end{equation}
It can be checked that \eqref{eq:dotVc_bound_raw} can be equivalently written as
\begin{equation}\label{eq:dot_Vc_bound_2}
\dot{V}_c\leq -2\beta V_c-Z,
\end{equation}
where we denote
\begin{align}
Z
=\ & b(x)
-\tilde x^TQ_1 \tilde x
+\tilde y^T Q_2\tilde y
+2c w(x)^T K^T\tilde y
+2c\tilde x^T Q_3^T \tilde y, \label{eq:Z_definition}\\
b(x) 
=\ &
\tilde x^T(\nabla f(x)-\nabla f(x^\ast)), \nonumber \\
w(x) 
=\ &
\frac{1}{2}\big(\nabla f(x)-\nabla f(x^\ast) 
+\nabla_x\Theta_\rho(x,\lambda)
-\nabla_x\Theta_\rho(x^\ast,\lambda)
-2\beta \tilde x\big), \nonumber
\end{align}
and
\begin{align*}
\tilde x &= x-x^\ast,
\qquad
\tilde y = \begin{bmatrix}
\nu-\nu^\ast \\
\lambda-\lambda^\ast
\end{bmatrix}, 
\qquad
K = \begin{bmatrix}
A \\
J
\end{bmatrix},
\\
Q_1 &=
cA^TA+\frac{c}{2}(J^T\hat{\Gamma}_{x,\lambda}\overline{J}_x+\overline{J}_x^T\hat{\Gamma}_{x,\lambda}J)+\beta I, \\
\tilde Q_{2}
&=
\frac{1}{2\rho}(I-\tilde\Gamma_\lambda)
+\frac{c}{2}(JJ^T\Gamma_\lambda+\Gamma_\lambda JJ^T)
-\beta I \\
Q_2 &=
\!\!\begin{bmatrix}
cAA^T-\beta I & \!\!\frac{c}{2}AJ^T\!(I \!+\! \Gamma_\lambda) \\
\frac{c}{2}(I \!+\! \Gamma_\lambda)JA^T\! & \tilde Q_2\\
\end{bmatrix}\!,
\ \ \ \ 
Q_3 \!=\!
\frac{1}{2\rho}\!\begin{bmatrix}
0 \\
(I\!-\!\Gamma_\lambda)J
\end{bmatrix}
\!.
\end{align*}

We now need to show $Z\geq 0$, which will then imply $\dot{V}_c\leq-2\beta V_c$ by \eqref{eq:dot_Vc_bound_2}. Without loss of generality we assume that $\mathcal{I}=\{1,2,\ldots,|\mathcal{I}|\}$. We first present the following lemma to give a positive definite lower bound of $Q_2$, whose proof is postponed to \ref{sec:proof_lemma_Q2}:
\begin{lemma}\label{lemma:Q2_lowerbound}
When \eqref{eq:main_thm_cond_1} is satisfied, we have
$$
Q_2\succeq \frac{c}{2}\begin{bmatrix}
AA^T & AJ_{\mathcal{I}}^T & \\
J_{\mathcal{I}}A^T & J_{\mathcal{I}}J_{\mathcal{I}}^T & \\
& & L_g^2 I
\end{bmatrix}.
$$
\end{lemma}

Lemma \ref{lemma:Q2_lowerbound} implies that $Q_2$ is positive definite as well as $Q_2^{-1}$. This allows us to reformulate $Z$ in \eqref{eq:Z_definition} as
$$
\begin{aligned}
Z
=\ &
b(x)-\tilde x^T Q_1\tilde x
-c^2\left\|Q_3\tilde x+Kw(x)\right\|_{Q_2^{-1}}^2 
+\left\|\tilde y+cQ_2^{-1}Q_3\tilde x+cQ_2^{-1}Kw(x)\right\|_{Q_2}^2.
\end{aligned}
$$
Now $Z\geq 0$ will follow directly from the following lemma:
\begin{lemma}\label{lemma:Zpositive_part2}
When the conditions \eqref{eq:main_thm_cond} are satisfied, we have
$$
b(x)-\tilde x^T Q_1\tilde x
-c^2\left\|Q_3\tilde x+Kw(x)\right\|_{Q_2^{-1}}^2\geq 0.
$$
\end{lemma}
The proof of Lemma \ref{lemma:Zpositive_part2} is presented in \ref{sec:proof_lemma_Z}. Now we have established the bound \eqref{eq:semi_exponen_stab}.

\vspace{2pt}
\noindent\textbf{Step 3: Prove Corollary~\ref{corollary:main}.}
For any initial distance $d_0$, define $\hat{\beta}(d_0)$ to be the largest number that satisfies the conditions \eqref{eq:main_thm_cond}. By utilizing the behavior of $\delta_{\min}$, $L_g$ and $M_{\Theta}$, we can verify that $\hat{\beta}(d_0)$ is a non-increasing function of $d_0$ that converges to zero as $d_0\rightarrow+\infty$. Then, since $M_\beta=\sqrt{\|P_c\|\|P_c^{-1}\|}$ converges to $1$ as $\beta\rightarrow 0$, there exists $\bar{\beta}>0$ such that $M_{\beta}\leq 2$ for $\beta\leq \bar{\beta}$. We choose $\bar{\beta}$ to be sufficiently small such that there exists $\bar{h}>0$ satisfying
$$
\inf_{d_0< \bar{h}}\hat{\beta}(d_0)\geq \bar{\beta}
\quad\textrm{and}\quad
\sup_{d_0> \bar{h}}\hat{\beta}(d_0)\leq \bar{\beta}.
$$

Now let $h\geq 0$ be arbitrary, and let $z(0)=(x(0),\lambda(0),\nu(0))$ be any initial point such that $\lambda(0)\geq 0$ and $d_0\leq h$. Without loss of generality we only consider the case $h\geq\bar{h}$. Then,
\begin{enumerate}
\item if $d_0\geq\bar h$, since $\hat{\beta}(h)\leq \hat{\beta}(d_0)\leq\bar\beta$, by \eqref{eq:semi_exponen_stab} we have
$$
\begin{aligned}
\|z(t)-z^\ast\|
\leq\ &
M_{\hat\beta(d_0)}\, e^{-{\hat\beta(d_0)}\cdot t}
\|z(0)-z^\ast\|
\leq
2 e^{-\hat{\beta}(h)\cdot t}\|z(0)-z^\ast\|;
\end{aligned}
$$
\item if $d_0< \bar{h}$, since $\bar\beta\leq\hat{\beta}(d_0)$ also satisfies the conditions \eqref{eq:main_thm_cond}, by \eqref{eq:semi_exponen_stab} we have
$$
\|z(t) \!-\! z^\ast\|
\leq
M_{\bar\beta}\, e^{-\bar\beta\cdot t}
\|z(0) \!-\! z^\ast\|
\leq
2 e^{-\hat{\beta}(h)\cdot t}\|z(0) \!-\! z^\ast\|.
$$
\end{enumerate}
In other words, we have $\|z(t)-z^\ast\|
\leq 2 e^{-\hat{\beta}(h)\cdot t}\|z(0)-z^\ast\|$ for all initial points satisfying $d_0\leq h$. This justifies semi-global exponential stability of Aug-PDGD.

\section{A Counterexample of Global Exponential Stability}\label{sec:counter_example}
In this section, we present an example of a smooth strongly convex optimization problem of which the Aug-PDGD does not achieve global exponential stability.

Consider the following problem:
\begin{equation}\label{eq:counter_eg_convex_program}
\begin{aligned}
\min_{x\in\mathbb{R}}\ \ \mfrac{1}{2}x^2
\qquad\textrm{s.t.} \ \ &
x\leq 1\ \ \textrm{and} \ \ x=0.
\end{aligned}
\end{equation}
It can be checked that this problem satisfies Assumptions 1--3, and the unique KKT point is $(x^\ast,\lambda^\ast,\nu^\ast)=(0,0,0)$. The corresponding Aug-PDGD is given by
\begin{equation}\label{eq:counter_eg}
\begin{aligned}
\dot{x}(t) &= -x(t) -\nu(t)
-[\rho(x(t)-1)+\lambda(t)]_+ , \\
\dot{\lambda}(t) &= \frac{1}{\rho}\left([\rho(x(t)-1)+\lambda(t)]_+-\lambda(t)\right), \\
\dot{\nu}(t) &=
x(t).
\end{aligned}
\end{equation}
We denote $z(t)\coloneqq(x(t),\lambda(t),\nu(t))$ and $z^\ast\coloneqq(x^\ast,\lambda^\ast,\nu^\ast)$ as usual.

\begin{proposition}
The equilibrium point $z^\ast=(0,0,0)$ of \eqref{eq:counter_eg} is not globally exponentially stable. In other words, there do not exist $M>0$ and $\xi>0$ such that for any initial point $z(0)=(x(0),\lambda(0),\nu(0))\in\mathbb{R}\!\times[0,+\infty)\!\times\!\mathbb{R}$, the solution $z(t)$ to \eqref{eq:counter_eg} satisfies
$$
\|z(t)-z^\ast\|
\leq M e^{-\xi t}\|z(0)-z^\ast\|,\qquad
\forall t\geq 0.
$$
\end{proposition}
\begin{proof}
Let $\alpha>0$ be arbitrary. It can be checked that
\begin{align*}
x(t) & =
\left\{
\begin{aligned}
& \mfrac{1}{2}, & & t\in[0,\alpha], \\
& \mfrac{\sqrt{3}}{3}
e^{-\frac{t\!-\!\alpha}{2}}\sin\!\left(\!
\mfrac{\sqrt{3}}{2}(t\!-\!\alpha)+\mfrac{\pi}{3}\!\right),
& & t>\alpha,
\end{aligned}
\right. \\
\lambda(t)&=
\left\{
\begin{aligned}
& \mfrac{1}{2}(\alpha+\rho-t), &\ \ \  \qquad\qquad\qquad & t\in[0,\alpha], \\
& \mfrac{1}{2}\rho\, e^{-\frac{t-\alpha}{\rho}},
& & t>\alpha,
\end{aligned}
\right. \\
\nu(t) & =
\left\{
\begin{aligned}
& \mfrac{1}{2}(t-\alpha-1), & & t\in[0,\alpha], \\
& \mfrac{\sqrt{3}}{3}
e^{-\frac{t\!-\!\alpha}{2}}\sin\!\left(\!
\mfrac{\sqrt{3}}{2}(t\!-\!\alpha)-\mfrac{\pi}{3}\!\right),
& & t>\alpha,
\end{aligned}
\right.
\end{align*}
gives a solution to \eqref{eq:counter_eg} with inital conditions $x(0)=1/2$, $\lambda(0)=(\alpha+\rho)/2$, $\nu(0)=-(\alpha+1)/2$. Therefore
$$
\begin{aligned}
h_\alpha(t) \coloneqq \,
&
\left.\left\|
z(t)-z^\ast
\right\|^2\right|_{z(0)=\left(\frac{1}{2},\frac{\alpha+\rho}{2},-\frac{\alpha+1}{2}\right)} \\
=\, &
\left\{
\begin{aligned}
& \mfrac{1}{4}
\left[1+(\alpha+\rho-t)^2
+(t-\alpha-1)^2\right], &\ \ & t\in[0,\alpha], \\
& \frac{\rho^2}{4}e^{-\!\frac{2}{\rho}(t\!-\!\alpha)}
+e^{-(t\!-\!\alpha)}
\frac{2\!+\!\cos\sqrt{3}(t\!-\!\alpha)}{6}, &&
t> \alpha,
\end{aligned}
\right.
\end{aligned}
$$
and we have
\begin{equation}\label{eq:counter_eg_proof_temp}
\int_0^\infty h_\alpha(t)\,dt
=\frac{\alpha^3}{6}+\frac{(\rho+1)\alpha^2}{4}+\frac{(2+\rho^2)\alpha}{4}+\frac{3+\rho^{3}}{8}.
\end{equation}
Now suppose there exist $M>0$ and $\xi>0$ such that
$$
\left\|
z(t)-z^\ast\right\|
\leq M e^{-\xi t}
\left\|z(0)-z^\ast\right\|,\quad\forall t\geq 0
$$
for all solutions to \eqref{eq:counter_eg} with any initial point $z(0)=(x(0),\lambda(0),\nu(0))\in\mathbb{R}\!\times[0,+\infty)\!\times\!\mathbb{R}$. We then have
$$
\begin{aligned}
\int_0^\infty h_\alpha(t)\,dt
\leq\,&
\int_0^\infty \! M^2 e^{-2\xi t}
\!\left(\frac{1+(\alpha\!+\!\rho)^2+(\alpha\!+\!1)^2}{4}\right)dt \\
=\,&
\left(\frac{1+(\alpha\!+\!\rho)^2+(\alpha\!+\!1)^2}{4}\right)\frac{M^2}{2\xi},
\qquad\qquad\forall \alpha>0.
\end{aligned}
$$
However, this contradicts \eqref{eq:counter_eg_proof_temp} for sufficiently large $\alpha$.
\end{proof}

It has been shown in \cite{qu2019exponential,dhingra2018proximal,ding2019global} that for smooth strongly convex optimization problems with affine inequality constraints $g(x)=Fx-\upsilon\leq 0$ and equality constraints $Ax=b$, the Aug-PDGD achieves global exponential stability when the row vectors of $F$ and $A$ are linearly independent. However, this is not the case for the example \eqref{eq:counter_eg_convex_program}: While LICQ still holds for \eqref{eq:counter_eg_convex_program}, the associated row vectors of the inactive affine inequality constraint $x\leq 1$ and the equality constraint $x=0$ are linearly dependent. This example demonstrates that linear independence of the row vectors of $F$ and $A$ is a key condition for achieving global exponential stability for Aug-PDGD.

\section{Numerical Example}\label{sec:simulation}
We consider a convex program that is abstracted from the optimal power curtailment of $n$ solar panels in a distribution feeder of $m$ buses. The problem is formulated as
$$
\begin{aligned}
\min_{p,q\in\mathbb{R}^n}\quad & \sum_{i=1}^{n} c_{p}\left(p_i-p^{\mathrm{PV}}_i\right)^2+c_q q_i^2 \\
\textrm{s.t.}\quad &
p_i^2+q_i^2\leq S_{\max,i}^2,\quad i=1,\ldots,n,\\
& 0 \leq p \leq p^{\mathrm{PV}},
\quad v_{\min}\leq Mp+Nq+r\leq v_{\max}.
\end{aligned}
$$
Here $p,q\in\mathbb{R}^n$ model the real and reactive power injections of the inverters connected to solar panels; $S_{\max}\in\mathbb{R}^n$ gives the rated apparent power of the inverters; $p^{\mathrm{PV}}\in\mathbb{R}^n$ gives the real power generated by solar panels; the map $(p,q)\mapsto Mp+Nq+r$ is derived from the DistFlow model \cite{baran1989optimal} that maps power injections to voltage magnitudes; $v_{\min},v_{\max}\in\mathbb{R}^m$ are bounds on voltage magnitudes; $c_p$ and $c_q$ are real positive constants.

The distribution feeder is a single-phase version of the IEEE 37-node test feeder \cite{schneider2018analytic}. Figure \ref{fig:network} shows the network topology and the locations where the solar panels are installed, and Table \ref{tab:inverters} gives the rated apparent power $S_{\max,i}$. We set $p^{\mathrm{PV}}_i$ to be proportional to $S_{\max,i}$ such that the total real power generation $\sum_i p^{\mathrm{PV}}_i$ is $4$ times the total real load, which models a scenario of very high penetration of solar generation. We set $c_p=3$ and $c_q=1$, and $v_{\min,j}=0.95$, $v_{\max,j}=1.05$ for each $j$. We scale the constraints $v_{\min}\leq Mp+Nq+r\leq v_{\max}$ by a factor of $2\times 10^2$, so that the nonzero entries of the optimal Lagrange multiplier $\lambda^\ast$ have approximately the same order of magnitude.

\begin{figure}
    \centering
    \includegraphics[width=.65\textwidth]{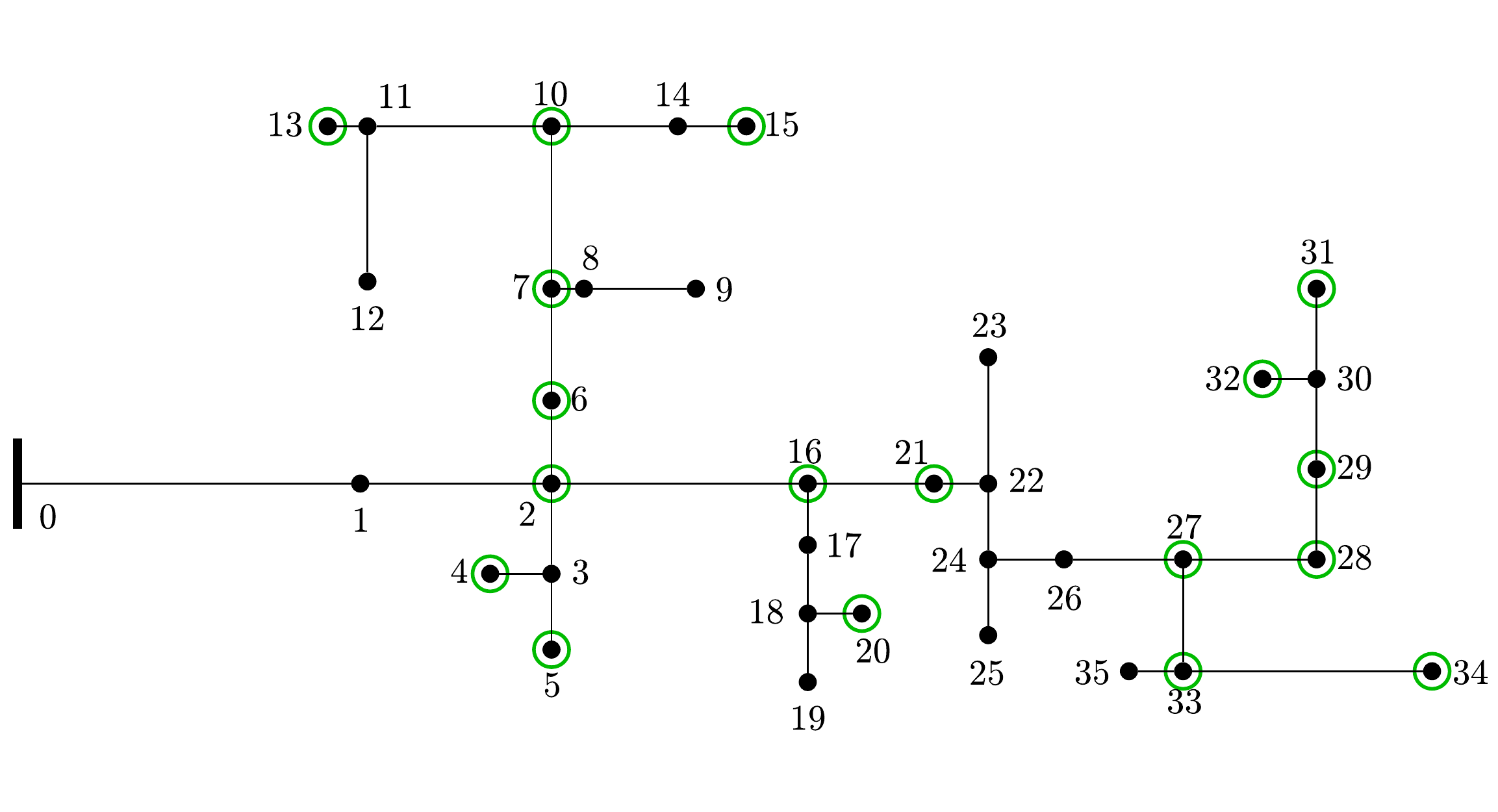}
    \caption{Topology of the distribution feeder. Solar panels and inverters are installed at buses marked by green hollow circles.}
    \label{fig:network}
\end{figure}

\begin{table}
    \centering
    \small
    \begin{tabular}{|c|c|c|c|c|c|c|}
    \hline
         Inverter ID & 1 & 2 & 3 & 4 & 5 & 6 \\
         \hline
         Bus No. & 2 & 4 & 5 & 6 & 7 & 10 \\
         \hline
         $S_{\max,i}$ (p.u.) & 2.7 & 1.35 & 2.7 & 1.35 & 
         2.025 & 2.025\\
         \hline
         \hline
         Inverter ID & 7 & 8 & 9 & 10 & 11 & 12 \\
         \hline
         Bus No. & 13 & 15 & 16 & 20 & 21 & 27 \\
         \hline
         $S_{\max,i}$ (p.u.) & 2.7 & 2.7 & 1.35 & 2.025 & 2.025 & 2.025\\
         \hline
         \hline
         Inverter ID & 13 & 14 & 15 & 16 & 17 & 18\\
         \hline
         Bus No. & 28 & 29 & 31 & 32 & 33 & 34 \\
         \hline
         $S_{\max,i}$ (p.u.) & 2.7 & 2.7 & 1.35 & 
         2.7 & 2.025 & 1.35\\
         \hline
    \end{tabular}
    \caption{Locations and rated apparent power $S_{\max,i}$ for each inverter.}
    \label{tab:inverters}
\end{table}

For the Aug-PDGD, we choose $\rho=0.1$. We simulated three cases, where the initial point $(x(0),\lambda(0))$ is selected randomly such that
$
d_0/\left\|
(x^\ast,
\lambda^\ast)
\right\|
$
is equal to $0.5$, $10$ and $50$ respectively.
Figure \ref{fig:results} shows the curves of the normalized distances
$
\left\|
(x(t)\!-\!x^\ast,
\lambda(t)\!-\!\lambda^\ast)
\right\|/\left\|
(x^\ast,
\lambda^\ast)
\right\|
$
as a function of time $t$, where each case consists of $10$ instances of randomly selected initial points $(x(0),\lambda(0))$. We see that while the distance $\left\|
(x(t)-x^\ast,
\lambda(t)-\lambda^\ast)
\right\|$ decreases exponentially on the whole, the exponential convergence rates differ for different $d_0$. Furthermore, for each single instance, the decreasing rate also changes as $(x(t),\lambda(t))$ approaches the KKT point. Especially, we observe that for the case $d_0=10\|(x^\ast,\lambda^\ast)\|$, the decreasing rates during $t\in(0,10)$ are smaller than those for $t>30$ where $(x(t),\lambda(t))$ finally achieves the same stable decreasing rate for all $10$ instances. These observations suggest that the numerical example may only achieve semi-global exponential stability.

\begin{figure}
    \centering
    \includegraphics[width=.6\textwidth]{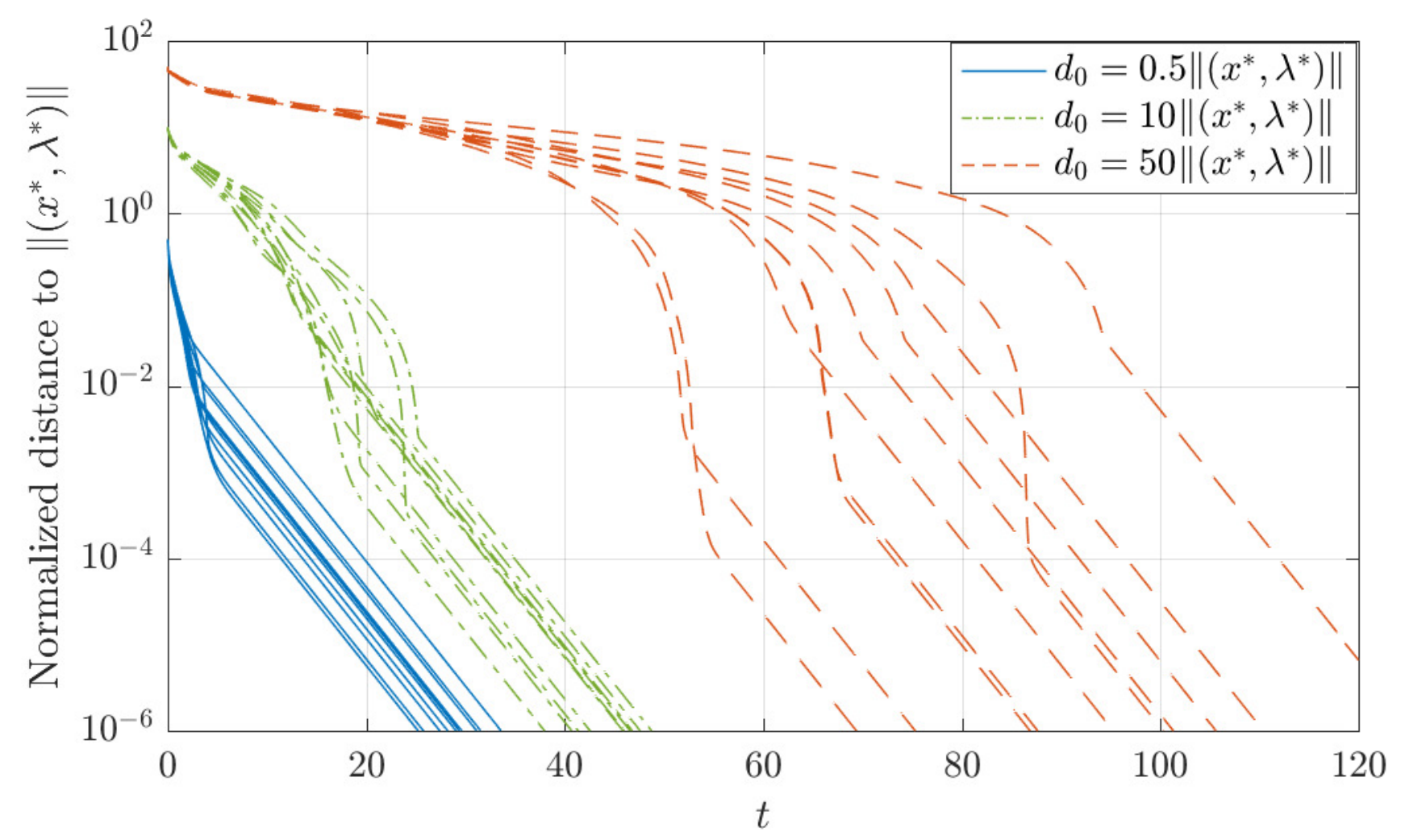}
    \caption{Illustration of the relative distances to $(x^\ast,\lambda^\ast)$ with respect to time $t$ for $30$ random instances.}
    \label{fig:results}
\end{figure}

\section{Conclusion}

This paper introduced the augmented primal-dual gradient dynamics (Aug-PDGD) for constrained convex optimization, and analyzed its stability behavior. Specifically, we extended the results in \cite{qu2019exponential} to more general settings with nonlinear constraints, and showed that the Aug-PDGD for smooth constrained convex optimization achieves semi-global exponential stability when the objective is strongly convex. We also presented an example showing that the Aug-PDGD may fail to achieve global exponential stability for general smooth strongly convex programs.

We point out a few possible extensions of this work.
\begin{enumerate}

\item The inequalities \eqref{eq:main_thm_cond} provide a conservative estimate on the true exponential convergence rate. How tight this estimate is remains an interesting open question.

\item We observe that the bound provided by \eqref{eq:main_thm_cond} may not be ``robust'' in the following sense: A small perturbation on $f(x)$ or $g(x)$ that inactivates an originally active constraint could lead to a sharp decrease in $\delta_{\min}$ and consequently a sharp decrease in $\beta$. We suspect that it is possible to extend the results if we make the stronger assumption that the set $\{\nabla g_i(x^\ast): g_i(x^\ast)\geq-\epsilon\}$ is linearly independent for some given $\epsilon>0$, which can possibly lead to a bound on $\beta$ that is more ``robust''.

\item We are also interested in investigating the performance of the Aug-PDGD in time-varying settings.
\end{enumerate}

\linespread{0.95}
\appendix
\section{Proof of Lemma \ref{lemma:Q2_lowerbound}}\label{sec:proof_lemma_Q2}

Obviously \eqref{eq:main_thm_cond_1} implies
$
c\leq \mfrac{2\delta_{\min}}{23\rho (L_g^2+\|A\|^2)}
$.
Noting that $\gamma_{\lambda,i}=\tilde{\gamma}_{\lambda,i}=1$ for $i\in\mathcal{I}$, we can partition the matrix $Q_2$ as
$$
Q_2=\begin{bmatrix}
Q_{2,\mathcal{I}\mathcal{I}} & Q_{2,\mathcal{I}\mathcal{I}^c} \\
Q_{2,\mathcal{I}\mathcal{I}^c}^T & Q_{2,\mathcal{I}^c\mathcal{I}^c}
\end{bmatrix},
$$
where we denote
\begin{align*}
Q_{2,\mathcal{I}\mathcal{I}}
=\ &
\begin{bmatrix}
cAA^T - \beta I & cAJ_{\mathcal{I}}^T \\
cJ_{\mathcal{I}}^TA & cJ_{\mathcal{I}}J_{\mathcal{I}}^T-\beta I
\end{bmatrix}, \\
Q_{2,\mathcal{I}\mathcal{I}^c}
=\ &
\frac{c}{2}\begin{bmatrix}
A \\
J_{\mathcal{I}}
\end{bmatrix}J_{\mathcal{I}^c}^T(I+\Gamma_{\lambda,\mathcal{I}^c}), \\
Q_{2,\mathcal{I}^c\mathcal{I}^c}
=\ &
\frac{1}{2\rho}(I-\tilde\Gamma_{\lambda,\mathcal{I}^c})
-\beta I 
+\frac{c}{2}(J_{\mathcal{I}^c}J_{\mathcal{I}^c}^T\Gamma_{\lambda,\mathcal{I}^c}
+\Gamma_{\lambda,\mathcal{I}^c}J_{\mathcal{I}^c}J_{\mathcal{I}^c}^T), \\
\Gamma_{\lambda,\mathcal{I}^c}
=\ &
\operatorname{diag}\left(\gamma_{\lambda,i}\right)_{i\in\mathcal{I}^c},\quad 
\tilde{\Gamma}_{\lambda,\mathcal{I}^c}
=\operatorname{diag}\left(\tilde\gamma_{\lambda,i}\right)_{i\in\mathcal{I}^c},
\end{align*}
and $J_{\mathcal{I}^c}$ is formed by the rows of $J$ whose indices are in $\mathcal{I}^c$.

By the definition of $\delta_{\min}$, we have $1-\tilde\Gamma_{\lambda,\mathcal{I}^c}\succeq \delta_{\min} I$ for all $t\geq 0$. Together with $I-\tilde\Gamma_\lambda\succeq I-\Gamma_{\lambda}$, it can be shown that
$$
\begin{aligned}
Q_{2,\mathcal{I}^c\mathcal{I}^c}
\succeq\ &
\left(\mfrac{1}{2\rho}-cL_g^2\right)\delta_{\min}I
+cL_g^2(I-\Gamma_{\lambda,\mathcal{I}^c}) \\
&
+\frac{c}{2}(J_{\mathcal{I}^c}J_{\mathcal{I}^c}^T\Gamma_{\lambda,\mathcal{I}^c}
+\Gamma_{\lambda,\mathcal{I}^c}J_{\mathcal{I}^c}J_{\mathcal{I}^c}^T)
-\beta I,
\end{aligned}
$$
By \cite[Lemma 6]{qu2019exponential}, we have $L_g^2(I-\Gamma_{\lambda,\mathcal{I}^c})
+\frac{1}{2}(J_{\mathcal{I}^c}J_{\mathcal{I}^c}^T\Gamma_{\lambda,\mathcal{I}^c}
+\Gamma_{\lambda,\mathcal{I}^c}J_{\mathcal{I}^c}J_{\mathcal{I}^c}^T)\succeq 0$, and so
$$
\begin{aligned}
Q_{2,\mathcal{I}^c\mathcal{I}^c}
&\succeq
\left(\mfrac{1}{2\rho}-cL_g^2\right)\delta_{\min}I-\beta I.
\end{aligned}
$$
Then by the definition of $\kappa$, we have
$$
\begin{aligned}
Q_{2,\mathcal{I}\mathcal{I}} -
\frac{c}{2}\begin{bmatrix}
AA^T & \!\! AJ_{\mathcal{I}}^T \\
J_{\mathcal{I}}A^T & \!\!  J_{\mathcal{I}}J_{\mathcal{I}}^T
\end{bmatrix} 
=\ &
\frac{c}{2}\left(\begin{bmatrix}
AA^T & \!\! AJ_{\mathcal{I}}^T \\
J_{\mathcal{I}}A^T & \!\!  J_{\mathcal{I}}J_{\mathcal{I}}^T
\end{bmatrix}-\frac{\kappa}{2}I\right)
\succeq
\frac{c}{4}\begin{bmatrix}
AA^T & AJ_{\mathcal{I}}^T \\
J_{\mathcal{I}}A^T & J_{\mathcal{I}}J_{\mathcal{I}}^T
\end{bmatrix}
\succ 0,
\end{aligned}
$$
and
$$
\begin{aligned}
&Q_{2,\mathcal{I}^c\mathcal{I}^c}-\frac{cL_g^2}{2}I
-Q_{2,\mathcal{I}\mathcal{I}^c}^T
\left(Q_{2,\mathcal{I}\mathcal{I}}-\frac{c}{2}\begin{bmatrix}
AA^T & AJ_{\mathcal{I}}^T \\
J_{\mathcal{I}}A^T & J_{\mathcal{I}}J_{\mathcal{I}}^T
\end{bmatrix}\right)^{-1}Q_{2,\mathcal{I}\mathcal{I}^c} \\
\succeq\ &
c\left(\left(\frac{1}{2c\rho}-L_g^2\right)\delta_{\min}
-\frac{\kappa}{4}-\frac{L_g^2}{2}\right)I \\
&
-c(I+\Gamma_{\lambda,\mathcal{I}^c})J_{\mathcal{I}^c}\begin{bmatrix}
A^T & J_{\mathcal{I}}^T
\end{bmatrix}
\begin{bmatrix}
AA^T & AJ_{\mathcal{I}}^T \\
J_{\mathcal{I}}A^T & J_{\mathcal{I}}J_{\mathcal{I}}^T
\end{bmatrix}^{-1}
\begin{bmatrix}
A \\ J_{\mathcal{I}}
\end{bmatrix}
J_{\mathcal{I}^c}^T
(I+\Gamma_{\lambda,\mathcal{I}^c}) \\
\succeq\ &
c\left(\frac{23}{4}(L_g^2+\|A\|^2)-\frac{\kappa}{4}-\frac{3L_g^2}{2}\right) I
-4cL_g^2 I\succeq 0,
\end{aligned}
$$
where we have used $\delta_{\min}\leq 1$, $\kappa\leq L_g^2+\|A\|^2$ and that $X^T(XX^T)^{-1}X\preceq I$ for a full row rank matrix $X$. By the Schur complement condition, we get the desired result.

\section{Proof of Lemma \ref{lemma:Zpositive_part2}}\label{sec:proof_lemma_Z}

By Lemma \ref{lemma:Q2_lowerbound},
$$
\begin{aligned}
&(Q_3\tilde x+Kw(x))^TQ_2^{-1}(Q_3\tilde x+Kw(x)) \\
\leq\ &
\frac{2}{c}(Q_3\tilde x \!+\! Kw(x))^T \!\!
\begin{bmatrix}
AA^T \!\! & \!\! AJ_{\mathcal{I}}^T \!\! & \\
J_{\mathcal{I}}A^T \!\! & \!\! J_{\mathcal{I}}J_{\mathcal{I}}^T \!\! & \\
& & \!\! L_g^2 I
\end{bmatrix}^{\!-\!1}
\!\!\!\!
(Q_3\tilde x \!+\! Kw(x)).
\end{aligned}
$$
We have
$$
\begin{aligned}
& w(x)^T K^T \begin{bmatrix}
AA^T & AJ_{\mathcal{I}}^T & \\
J_{\mathcal{I}}A^T & J_{\mathcal{I}}J_{\mathcal{I}}^T & \\
& & L_g^2 I
\end{bmatrix}^{-1} Kw(x) \\
=\ &
w(x)^T
\begin{bmatrix}
A \\ J_{\mathcal{I}} \\ J_{\mathcal{I}^c}
\end{bmatrix}^T
\begin{bmatrix}
\begin{bmatrix}
AA^T & AJ_{\mathcal{I}}^T \\
J_{\mathcal{I}}A^T & J_{\mathcal{I}}J_{\mathcal{I}}^T
\end{bmatrix}^{-1}\!\!\!\! & \\
 & \!\! L_g^{-2} I
\end{bmatrix}
\begin{bmatrix}
A \\ J_{\mathcal{I}} \\ J_{\mathcal{I}^c}
\end{bmatrix}
w(x) \\
\leq\ &
w(x)^T
\left(
I+
L_g^{-2}J_{\mathcal{I}^c}^T J_{\mathcal{I}^c}
\right) w(x)
\leq 2\|w(x)\|^2,
\end{aligned}
$$
and
$$
\begin{aligned}
Q_3^T
\begin{bmatrix}
AA^T & AJ_{\mathcal{I}}^T & \\
J_{\mathcal{I}}A^T & J_{\mathcal{I}}J_{\mathcal{I}}^T & \\
& & L_g^2 I
\end{bmatrix}^{-1}
Q_3
=
\frac{1}{4\rho^2L_g^2}
J_{\mathcal{I}^c}^T(\Gamma_{\lambda,\mathcal{I}^c}-I)^2 J_{\mathcal{I}^c}
\preceq
\frac{1}{4\rho^2}I,
\end{aligned}
$$
and
$$
\begin{aligned}
w(x)^TK^T
\begin{bmatrix}
AA^T & AJ_{\mathcal{I}}^T & \\
J_{\mathcal{I}}A^T & J_{\mathcal{I}}J_{\mathcal{I}}^T & \\
& & L_g^2 I
\end{bmatrix}^{-1}
Q_3\tilde x
=\ &
\frac{1}{2\rho L_g^2}w(x)^T
J_{\mathcal{I}^c}^T
(I-\Gamma_{\lambda,\mathcal{I}^c})J_{\mathcal{I}^c}\tilde x \\
\leq\ &
\frac{1}{2\rho L_g^2}
\|w(x)\|\|J_{\mathcal{I}^c}^T
(I-\Gamma_{\lambda,\mathcal{I}^c})J_{\mathcal{I}^c}\tilde x\|
\leq
\frac{1}{2\rho}\|w(x)\|\|\tilde x\|.
\end{aligned}
$$
Therefore
$$
\begin{aligned}
& b(x)-\tilde x^T Q_1\tilde x-c^2\left\|Q_3\tilde x+Kw(x)\right\|_{Q_2^{-1}}^2 \\
\geq\ &
b(x) \!-\! \tilde x^TQ_1\tilde x
\!-\! 2c \!\left(
2\|w(x)\|^2
\!+\! \mfrac{1}{4\rho^2}\|\tilde x\|^2
\!+\! \mfrac{1}{\rho}\|w(x)\|\|\tilde x\|
\!\right)\!.
\end{aligned}
$$
Since $f(x)+\Theta_\rho(x,\lambda)$ is convex in $x$ and its gradient with respect to $x$ is $(\ell+M_{\Theta})$-Lipschitz in $x$, we have (see \cite[Theorem 2.1.5]{nesterov2004introductory})
$$
\begin{aligned}
& (\nabla f(x) \!-\! \nabla f(x^\ast)
\!+\! \nabla_x\Theta_\rho(x,\lambda)
\!-\! \nabla_x\Theta_\rho(x^\ast,\lambda))^T(x \!-\! x^\ast) \\
\geq
&\frac{1}{\ell \!+\! M_{\Theta}}\!\|\nabla f(x) \!-\! \nabla f(x^\ast)
\!+\! \nabla_x\Theta_\rho(x,\lambda)
\!-\! \nabla_x\Theta_\rho(x^\ast,\lambda)\|^2.
\end{aligned}
$$
In addition, \eqref{eq:main_thm_cond_2} implies $\kappa\mu/(4\beta)> \|A\|^2 + L_g^2$, and so
$$
\begin{aligned}
\beta<\frac{\kappa\mu}{4(L_g^2+\|A\|^2)}
\leq\frac{\ell}{4} \leq \frac{\ell+M_\Theta}{4},
\end{aligned}
$$
where we have used $\kappa\leq L_g^2+\|A\|^2$ and $\mu\leq\ell$. Therefore we can bound $\|w(x)\|$ by
\begin{flushleft}
\begin{equation*}
\begin{aligned}
4\|w(x)\|^2\!\!
=&
\|\nabla\! f(x) \!-\! \nabla\! f(x^\ast\!)
\!+\! \nabla_x\Theta_\rho(x,\!\lambda)
\!-\! \nabla_x\Theta_\rho(x^\ast\!,\!\lambda)\|^2
\\
&
\!-\!4\beta\tilde x^T
\big[\nabla\! f(x) \!-\! \nabla\! f(x^\ast\!) 
\!+\! \nabla_x\Theta_\rho(x,\!\lambda)
\!-\! \nabla_x\Theta_\rho(x^\ast\!,\!\lambda)\big]
\!+\! 4\beta^2\|\tilde x\|^2 \\
\leq\ &
(\ell+M_{\Theta}-4\beta)\tilde x^T\big[\nabla f(x)-\nabla f(x^\ast)
+\nabla_x\Theta_\rho(x,\lambda)
-\nabla_x\Theta_\rho(x^\ast,\lambda)\big]
+4\beta^2 \|\tilde x\|^2
 \\
\leq\ &
(\ell+M_{\Theta})
(\nabla f(x)-\nabla f(x^\ast))^T\tilde x
+ (\ell+M_{\Theta})M_{\Theta}\|\tilde x\|^2
+ 4\beta^2\|\tilde x\|^2,
\nonumber
\end{aligned}
\end{equation*}
where the first inequality can also be relaxed by
\end{flushleft}
$$
\begin{aligned}
&(\ell+M_{\Theta}-4\beta)\tilde x^T\big[\nabla f(x)-\nabla f(x^\ast)
+\nabla_x\Theta_\rho(x,\lambda)
-\nabla_x\Theta_\rho(x^\ast,\lambda)\big]
+4\beta^2 \|\tilde x\|^2 \\
\leq\ 
&(\ell+M_{\Theta}-4\beta)
(\ell+M_{\Theta})\|\tilde x\|^2
+ 4\beta^2\|\tilde x\|^2 \\
= \ &
(\ell+M_{\Theta}-2\beta)^2\|\tilde x\|^2
\leq (\ell+M_{\Theta})^2\|\tilde x\|^2,
\end{aligned}
$$
which shows that $\|w(x)\|\leq (\ell+M_{\Theta})\|\tilde x\|/2$. It's not hard to see that
$Q_1\preceq c(\|A\|^2+L_g^2+\mfrac{\kappa}{4})I$. 
Now we can get
$$
\begin{aligned}
& b(x)-\tilde x^T Q_1\tilde x-c^2\|Q_3\tilde x+Kw(x)\|_{Q_2^{-1}}^2 \\
\geq\,&
(1-c(\ell+M_{\Theta}))\left(\nabla f(x)-\nabla f(x^\ast)\right)^T\tilde x \\
&
- \!
c\left(\! \|A\|^2 + L_g^2+\mfrac{\kappa}{4} \right)\|\tilde x\|^2
\!- \!
c\left((\ell+M_{\Theta})M_{\Theta} \!+\! 4\beta^2\right)\|\tilde x\|^2
-\mfrac{c}{2\rho^2}\|\tilde x\|^2
-\mfrac{c}{\rho}(\ell+M_{\Theta})\|\tilde x\|^2 \\
\geq\, &
\|\tilde x\|^2 \!\Big[\mu(1\!-\!c(\ell+M_{\Theta}))
\!-\!
c \!
\left(\!
\|A\|^2 \!+\!
L_g^2
\!+\!
\mfrac{\kappa}{4}\!+\!4\beta^2
\!+\!
(\ell \!+\! M_{\Theta})(M_{\Theta} \!+\! 1/\rho)
\!+\!
\mfrac{1}{2\rho^2}\right)\!\!\Big].
\end{aligned}
$$
By \eqref{eq:main_thm_cond_2} and $c=4\kappa^{-1}\beta$, we get the desired inequality.

\bibliographystyle{unsrt}
\bibliography{biblio-arxiv}
\end{document}